\documentclass[11pt,reqno]{amsart}
\usepackage{enumerate}
\usepackage{amsfonts, amsmath, amssymb, amscd, amsthm, bm}

\usepackage{enumerate}

\usepackage{multicol}
\usepackage{comment}

 \newtheorem{thm}{Theorem}[section]

 \newtheorem{lem}[thm]{Lemma}
 
 \theoremstyle{definition}
 
 \theoremstyle{theorem}
 \newtheorem{rem}[thm]{Remark}

 \theoremstyle{claim}
 \numberwithin{equation}{section}

\def\R{\mathbb{R}^{N_1}}

\def\et{\tilde{e}}
\def\th{\tilde{h}}
\def\bn{\overline{\nabla}}
\def\la{\lambda}
\def\ep{\varepsilon}
\def\l{\langle}
\def\r{\rangle}
\def\tr{\mbox{tr}}
\newcommand{\s}{\mathbb{S}}

\numberwithin{equation}{section}

\newcounter{rom}
\renewcommand{\therom}{(\roman{rom})}
{\end{list}}

\title{On stable compact minimal submanifolds of Riemannian product manifolds}
\begin{document}

\author{Hang Chen}
\address {Department of Mathematical Sciences\\Tsinghua
University, Beijing 100084, People's Republic of China.}
\email{chenhang08@mails.tsinghua.edu.cn}

\author{Xianfeng Wang}
\address {School of Mathematical Sciences and LPMC\\Nankai
University, Tianjin 300071, People's Republic of China.}
\email {wangxianfeng@nankai.edu.cn}
\date{September 27, 2012}
\thanks {The first author was supported in part by NSFC Grant No. 11271214, The second author was supported in part by NSFC Grant No. 11171175, NSFC Grant No. 11201243 and ``the Fundamental Research Funds for the Central Universities".}


%%% ----------------------------------------------------------------------

\begin{abstract}
In this paper, we prove a classification theorem for the stable compact minimal submanifolds of
the Riemannian  product of an $m_1$-dimensional ($m_1\geq3$) hypersurface $M_1$ in the Euclidean space  and any Riemannian manifold $M_2$, when the sectional curvature $K_{M_1}$ of $M_1$ satisfies $\frac{1}{\sqrt{m_1-1}}\leq K_{M_1}\leq 1.$ This gives a generalization to the results of F. Torralbo and F. Urbano \cite{TU10}, where they obtained a classification theorem for  the stable minimal submanifolds of the Riemannian  product of a sphere  and any Riemannian manifold.
In particular, when the ambient space is an $m$-dimensional ($m\geq3$) complete  hypersurface $M$ in the Euclidean space,  if the sectional curvature $K_{M}$ of $M$ satisfies $\frac{1}{\sqrt{m+1}}\leq K_{M}\leq 1$, then we conclude that there exist no stable compact minimal
submanifolds in $M$.
\end{abstract}

\maketitle

\section{Introduction}
It is well-known that minimal submanifolds in a Riemannian manifold are
the critical points of the volume functional.
It is natural and important to study whether a given minimal submanifold  is stable or not, that is,
the considered minimal submanifold attains a local minimum of the volume functional or not.
In fact, we call a compact minimal submanifold $\Sigma$ in a Riemannian manifold $M$ \textit{stable}
if the second variation of the volume is nonnegative for every deformation of $\Sigma$.
The existence or non-existence of a stable minimal
submanifold is very closely related to the topological and Riemannian structures of the ambient space.
In this setting, Simons \cite{Simon68}  proved the following remarkable result on the non-existence of the stable minimal
submanifolds in the sphere.

\begin{thm}[Simons, \cite{Simon68}]\label{thm1.1}
There exist no stable compact minimal submanifolds in the Euclidean sphere $\mathbb{S}^n$.
\end{thm}

Lawson and Simons \cite{LS73} classified all the  stable compact minimal submanifolds of the complex projective space. Ohnita \cite{O86m} completed
the classification of  the stable compact minimal submanifolds in all the other compact rank one symmetric spaces (the real
projective space $P^n(\mathbb{R})$,  the quaternionic projective space $P^n(\mathbb{H})$
and the Cayley projective space $P^2(Cay)$). Besides,  there have been many works on the stability, instability and the index of
minimal  submanifolds in some other different ambient Riemannian spaces (see \cite{FC}, \cite{FCS}, \cite{O}, \cite{R}, \cite{U90}, \cite{U93}, \cite{U07}, etc.). However, among these results, only a few particular situations have been considered in arbitrary codimensional case.

It is remarkable that very recently Torralbo and Urbano (see \cite{TU10}) proved a classification theorem for the stable compact minimal submanifolds of the Riemannian product
of a sphere $\mathbb{S}^{m}(r)$ and any Riemannian manifold $M$.
They proved that

\begin{thm}[Torralbo and Urbano, \cite{TU10}]\label{ThmTU}
Let $M$ be any Riemannian manifold and  $\Phi=(\phi,\psi): \Sigma \rightarrow \s^{m}(r) \times M$  be a minimal immersion of a compact $n$-manifold  $\Sigma$, $n\geq 2$, satisfying either $m\geq 3$ or $m=2$ and $\Phi$ is a hypersurface. Then, $\Phi$ is stable if and only if
\begin{enumerate}
	\item$\Sigma=\s^{m}(r)$ and $\Phi(\Sigma)$ is a slice $\s^{m}(r)\times\{q\}$ with $q$ a point of $M$.
	\item$\Sigma$ is a covering of $M$ and $\Phi(\Sigma)$ is a slice $\{p\}\times M$ with $p$ a point of $\s^m(r)$.
	\item$\psi:\Sigma\rightarrow M$ is a stable minimal submanifold and $\Phi(\Sigma)$ is $\{p\}\times \psi(\Sigma)$ with $p$ a point of $\s^m(r)$.
	\item$\Sigma=\s^{m}(r)\times \hat{\Sigma}$, $\Phi=Id\times\psi$, and $\psi:\hat{\Sigma}\rightarrow M$ is a stable minimal submanifold.
\end{enumerate}
\end{thm}

The motivation of this paper is to  generalize Torralbo and Urbano's results to a Riemannian product of a $\delta$-pinched hypersurface $M_1$ in the Euclidean space and an arbitrary
Riemannian manifold $M_2$.  Note that
a Riemannian manifold $M$ is called $\delta$-pinched ($0<\delta\leq 1$) if the sectional curvature satisfies $\delta a\leq K_M\leq a$ everywhere for some positive number $a$, and one may take $a=1$ without loss of generality. Perhaps one of the most surprising facts in this paper is that our $\delta$ is
 a strictly monotonically decreasing function of $m_1$ (=  $\dim M_1$), and $\delta$ converges to $0$ as $m_1$ tends to infinity.

More precisely, we prove the following classification theorem for the stable compact minimal submanifolds of
the Riemannian  product of a hypersurface $M_1$ in the Euclidean space  and any Riemannian manifold $M_2$, with arbitrary codimensions.

\begin{thm}\label{thm1.4}
Let $\Phi=(\phi, \psi): \Sigma\to M=M_1\times M_2$  be a minimal immersion of an $n$-dimensional $(n\geq 2)$ compact manifold  $\Sigma$ into $M$,
where $M_1$ is a complete connected hypersurface in $\mathbb{R}^{m_1+1} (m_1\geq3)$ and $M_2$ is any Riemannian manifold. Assume that the sectional curvature $K_{M_1}$  of $M_1$ satisfies
$$\frac{1}{\sqrt{m_1-1}}\leq K_{M_1}\leq1.$$
Then, $\Phi$ is stable if and only if
\begin{enumerate}
\item$\Sigma=M_1$ and $\Phi(\Sigma)$ is a slice $M_1\times \{p_2\}$ with $p_2$ a point of $M_2$.
\item$\Sigma$ is a covering of $M_2$ and $\Phi(\Sigma)$ is a slice $\{p_1\}\times M_2$ with $p_1$ a point of $M_1$.
\item$\psi: \Sigma\to M_2$ is a stable minimal submanifold and $\Phi(\Sigma)$ is $\{p_1\}\times\psi(\Sigma)$ with $p_1$ a point of $M_1$.
\item$\Sigma=M_1\times\hat{\Sigma}$, $\Phi=Id\times\psi$, and $\psi: \hat{\Sigma}\to M_2$ is a  stable minimal submanifold.
\end{enumerate}

In particular, when the ambient space is  an $m$-dimensional $(m\geq 3)$ complete  hypersurface $M$  in  $\mathbb{R}^{m+1}$, if the sectional curvature $K_M$ of $M$satisfies
$$\frac{1}{\sqrt{m+1}}\leq K_{M}\leq 1 ,$$ then there exist no stable compact minimal submanifolds in $M$.
\end{thm}

\begin{rem}
Let $M_1$ be an ellipsoid (see \cite{SH01}) in the Euclidean space $\mathbb{R}^{m_1+1}:$
$$
M_1=\{(x_1,\ldots,x_{m_1+1})\in\mathbb{R}^{m_1+1}: \frac{x_1^2}{a_1^2}+\cdots+\frac{x_{m_1+1}^2}{a_{m_1+1}^2}=1\},
$$
with $0<a_1\leq a_2\leq \cdots\leq a_{m_1+1}$,
then the minimal and maximal principal curvatures of $M_1$ in $\mathbb{R}^{m_1+1}$ are given by
$$\lambda_{\min}=\frac{a_1}{a_{m_1+1}^2},~\lambda_{\max}=\frac{a_{m_1+1}}{a_1^2}.$$
Hence by choosing suitable $a_i$, we obtain a family of examples of $M_1$ which satisfy the conditions  in Theorem \ref{thm1.4}.
\end{rem}

\section{Preliminaries}
In this section, we consider $M_1$ as a
submanifold of the Euclidean space, and we get a key lemma for the proof of Theorem \ref{thm1.4}.
We use the same idea as in \cite{LS73}, \cite{O86m}, \cite{Simon68} and \cite{TU10}, etc., that is, taking the normal components of parallel vector fields of the Euclidean space where $M_1$ sits as the test sections.

We make the following convention on the ranges of indices:
\begin{equation*}
\begin{array}{rll}
1\leq i, j, k \leq n;&
n+1\leq\alpha, \beta, \gamma \leq n+p=m_1+m_2;\\
1\leq r, s, t \leq m_1;&
m_1+1\leq \mu, \nu \leq N_1.
\end{array}
\end{equation*}

Let $\Phi=(\phi, \psi): \Sigma\to M=M_1\times M_2$ be a minimal immersion of an $n$-dimensional compact manifold $\Sigma~(n\geq 2)$ into a Riemannian
 product manifold $M=M_1\times M_2$, where $M_i$ is an $m_i$-dimensional $(i=1, 2)$  Riemannian manifold and $m_1+m_2> n$. Denote by $p=\mbox{codim}(\Sigma)=m_1+m_2-n.$
We choose a local orthonormal frame $\{e_1,\ldots,e_{n+p}\}$ in $M$ such that $\{e_1,\ldots,e_n\}$ is an orthonormal frame in $T\Sigma\subset TM$ and $\{e_{n+1},\ldots,e_{n+p}\}$ is an orthonormal frame in $T^{\perp}\Sigma\subset TM.$
Since $\Phi$ is a minimal immersion,
it is well-known that the Jacobi operator $J$ of the second variation is a strongly elliptic operator acting on the sections of the normal
bundle of $\Phi$. For any $\eta\in\Gamma(T^{\bot}\Sigma)$, $J$ is given by (see \cite{TU10})

$$
J\eta=-(\Delta^{\bot}+\mathfrak{B}+\mathfrak{R})\eta,
$$
where
$$
\Delta^{\bot}=\sum_{i=1}^n\{\nabla^{\bot}_{e_i}\nabla^{\bot}_{e_i}-
\nabla^{\bot}_{\nabla_{e_i}e_i}\},\quad
\mathfrak{B}(\eta)=\sum_{i=1}^nh(e_i,A_{\eta}e_i),\quad \mathfrak{R}(\eta)=\sum_{i=1}^n(\bar{R}(\eta,e_i)e_i)^{\bot}.
$$
Here $\nabla^{\bot}$ is the normal connection, $h$ is the second fundamental form of $\Phi$, $A_{\eta}$ is the shape operator of $\Phi$,
 $\bar{R}$ is  the curvature tensor on $M=M_1\times M_2$ and $\bot$ denotes normal component. Then the immersion $\Phi$ is {\it stable} if and only if
$$
Q(\eta)=\int_{\Sigma}\langle J\eta,\eta\rangle\,d\Sigma\geq 0,\quad \forall ~\eta\in\Gamma(T^{\bot}\Sigma).
$$

Now let $M_1$ be an $m_1$-dimensional compact submanifold in the Euclidean space $\R$, and $M_2$ be any Riemannian manifold. We have the following immersions:
$$\Sigma\to M=M_1\times M_2\to \mathbb{R}^{N_1}\times M_2.$$
For any tangent vector $v$ at a point $p=(p_1, p_2)\in M_1\times M_2$, we have the decomposition $v=(v^1, v^2)$, where $v^i$ is tangent to $M_i$ at $p_i$.
Let $\{\et_{\mu}\}=\{\et_{m_1+1},\ldots,\et_{N_1}\}$ be an orthonormal frame in $T^{\perp}M_1 \subset T\R$. For convenience, we identify the tangent vector $u\in TM_1$ with $(u,0)\in T^{\perp}M\subset T(\R\times M_2).$
We denote by
$\nabla$ , $\bn$  and $D$ the Levi-Civita connections on $\Sigma$, $M$ and $\R\times M_2$, respectively.
Let $h$, $\th$ and $B_1$ denote the second fundamental forms of $\Sigma\to M$, $M\to\mathbb{R}^{N_1}\times M_2$ and $M_1\to\mathbb{R}^{N_1}$,
respectively. Denote by $R^1$ the curvature tensor on $M_1$, the Gauss equation of $M_1$ in $\R$ is given by
\begin{equation}
\label{Gauss}
\l R^1(X,Y)Z,W\r=\l B_1(X,W),B_1(Y,Z)\r-\l B_1(X,Z),B_1(Y,W)\r.
\end{equation}

For a fixed vector $U\in \mathbb{R}^{N_1},$ by identifying $U$ with $(U, 0)\in T(\R\times M_2)$, we decompose $U$ as follows:
$$U=T_U+N_U+\sum_{\mu=1}^{N_1}\l U, \et_{\mu}\r \et_{\mu},$$
where $T_U=\sum\limits_{j=1}^n\l U, e_j\r e_j$ is tangent to $\Sigma$ and $N_U=\sum\limits_{\beta=n+1}^{n+p}\l U, e_{\beta}\r e_{\beta}$ is normal to $\Sigma$ in $M$.
By deriving the fixed vector $U\in \mathbb{R}^{N_1}$ with respect to $e_i$, we obtain that
\begin{equation}
\begin{aligned}\label{DeiU}
0=&D_{e_i}U
=D_{e_i}T_U+D_{e_i}N_U+D_{e_i}(\sum_{\mu}\l U, \et_{\mu}\r \et_{\mu})\\
=&\bn_{e_i}T_U+\th(e_i, T_U)+\bn_{e_i}N_U+\th(e_i, N_U)+\sum_{\mu}\l U, D_{e_i}\et_{\mu}\r \et_{\mu}+\sum_{\mu}\l U, \et_{\mu}\r D_{e_i}\et_{\mu}\\
=&\nabla_{e_i}T_U+h(e_i, T_U)+\th(e_i, T_U)+\nabla^{\perp}_{e_i}N_U-A_{N_U}(e_i)+\th(e_i, N_U)\\
&+\sum_{\mu}\l U, D_{e_i}\et_{\mu}\r \et_{\mu}+\sum_{\mu}\l U, \et_{\mu}\r D_{e_i}\et_{\mu}.
\end{aligned}
\end{equation}
Taking the tangent and normal parts of \eqref{DeiU} respectively, we obtain that
\begin{equation}
\begin{aligned}\label{tan}
\nabla_{e_i}T_U=A_{N_U}(e_i)-\sum_{j, \mu}\l U, \et_{\mu}\r \l D_{e_i}\et_{\mu}, e_j\r e_j
=A_{N_U}(e_i)+\sum_{j}\l \th(e_i,e_j),U\r e_j,
\end{aligned}
\end{equation}
\begin{equation}
\begin{aligned}\label{nor}
\nabla^{\perp}_{e_i}N_U=&-h(e_i, T_U)-\sum_{\beta, \mu}\l U, \et_{\mu}\r \l D_{e_i}\et_{\mu}, e_{\beta}\r e_{\beta}\\
=&-h(e_i, T_U)+\sum_{\beta, \mu}\l U, \et_{\mu}\r \l \th(e_i, e_{\beta}),\et_{\mu}\r e_{\beta}\\
=&-h(e_i, T_U)+\sum_{\beta, \mu}\l U, \et_{\mu}\r \l B_1(e_i^1, e_{\beta}^1),\et_{\mu}\r e_{\beta}.
\end{aligned}
\end{equation}
We choose a local orthonormal frame $\{e_1,\ldots,e_n\}$ in a neighborhood of the point $p$ such that $\nabla_{e_i}e_j(p)=0$. By deriving \eqref{nor} again, using \eqref{tan}-\eqref{nor}, the Codazzi equation and the minimality, we have
\begin{equation}\label{laplace}
\begin{aligned}
\Delta^{\bot}N_U
=&\sum\limits_i(\bar{R}(T_U,e_i)e_i)^{\bot}-\sum\limits_ih(e_i,A_{N_U}e_i)-\sum\limits_{i,j}\l \th(e_i,e_j),U\r h(e_i,e_j)\\
+&\sum_{i,\beta,\mu}\l U,D_{e_i}\et_{\mu}\r\l B_1(e_i^1, e_{\beta}^1),\et_{\mu}\r e_{\beta}+\sum_{i,\beta,\mu}\l U, \et_{\mu}\r\nabla^{\perp}_{e_i}(\l B_1(e_i^1, e_{\beta}^1),\et_{\mu}\r e_{\beta}).
\end{aligned}
\end{equation}
From the definition of the Jacobi operator $J$ and \eqref{laplace}, we obtain that
\begin{equation}\label{JNU}
\begin{aligned}
-JN_U=&\sum\limits_i(\bar{R}(T_U+N_U,e_i)e_i)^{\bot}-\sum\limits_{i,j}\l \th(e_i,e_j),U\r h(e_i,e_j)\\
+&\sum_{i,\beta,\mu}\l U,D_{e_i}\et_{\mu}\r\l B_1(e_i^1, e_{\beta}^1),\et_{\mu}\r e_{\beta}+\sum_{i,\beta,\mu}\l U, \et_{\mu}\r\nabla^{\perp}_{e_i}(\l B_1(e_i^1, e_{\beta}^1),\et_{\mu}\r e_{\beta}).
\end{aligned}
\end{equation}
On the other hand, since $T_U+N_U\in TM_1$, from the Gauss equation \eqref{Gauss}, we have
\begin{equation}\label{RU2}
\begin{aligned}
&\sum\limits_i(\bar{R}(T_U+N_U,e_i)e_i)^{\bot}\\=&\sum_{i,\beta}\l\bar{R}(T_U+N_U, e_i)e_i, e_{\beta}\r e_{\beta}
=\sum_{i,\beta}\l R^{1}(T_U+N_U, e_i^1)e_i^1, e_{\beta}^1\r e_{\beta}\\
=&\sum_{i,\beta}\Big(\l B_1(T_U+N_U, e_{\beta}^1), B_1(e_{i}^1, e_{i}^1)\r -\l B_1(T_U+N_U, e_{i}^1), B_1(e_{\beta}^1, e_{i}^1)\r \Big)e_{\beta}\\
=&\sum_{i, 1\leq A\leq n+p,\beta}\l U, e_A^1\r\Big(\l B_1(e_A^1, e_{\beta}^1), B_1(e_{i}^1, e_{i}^1)\r -\l B_1(e_A^1, e_{i}^1), B_1(e_{\beta}^1, e_{i}^1)\r\Big)e_\beta,
\end{aligned}
\end{equation}
where we used the fact that $\{e_1^1,\ldots,e_{n+p}^1\}$ is a spanning set of $T_{p_1}M_1$.

Let $\{ E_1, \cdots, E_{N_1}\}$ be an orthonormal basis of $\R$. We define  $$F=-\sum\limits_{A=1}^{N_1}\l N_{E_A}, JN_{E_A}\r,$$ from \eqref{JNU} and \eqref{RU2} we obtain that
\begin{equation}\label{sum}
\begin{aligned}
F
=&\sum_{i, 1\leq A\leq n+p,\beta}\l  e_A^1,e_\beta^1\r\Big(\l B_1(e_A^1, e_{\beta}^1), B_1(e_{i}^1, e_{i}^1)\r -\l B_1(e_A^1, e_{i}^1), B_1(e_{\beta}^1, e_{i}^1)\r\Big)\\
&+\sum_{i,\beta,\mu}\l e_{\beta},D_{e_i}\et_{\mu}\r\l B_1(e_i^1, e_{\beta}^1),\et_{\mu}\r \\
=&\sum_{i,\beta}\Big(\l B_1(e_\beta^1, e_{\beta}^1), B_1(e_{i}^1, e_{i}^1)\r -\l B_1(e_\beta^1, e_{i}^1), B_1(e_{\beta}^1,e_{i}^1)\r-||B_1(e_i^1, e_{\beta}^1)||^2\Big)\\
=&\sum_{i,\beta}\Big(\l B_1(e_\beta^1, e_{\beta}^1), B_1(e_{i}^1, e_{i}^1)\r-2||B_1(e_i^1, e_{\beta}^1)||^2\Big).
\end{aligned}
\end{equation}

When $M_1$  is a compact submanifold in a Euclidean sphere $\mathbb{S}^{N_1}(c)$ with constant sectional curvature $c>0$, we can still consider $M_1$ as a submanifold in the Euclidean space, as $\mathbb{S}^{N_1}(c)$ is a totally umbilical submanifold in the Euclidean space with unit normal $\nu$.
We have $\Phi=(\phi,\psi):\Sigma\to M_1\times M_2\to \mathbb{S}^{N_1}(c)\times M_2\to \mathbb{R}^{N_1+1}\times M_2$.
We denote by $B_1$ the second fundamental form of the immersion $f_1:M_1\to \mathbb{S}^{N_1}(c)$ and $B_0$ the second fundamental form of $M_1\to \mathbb{R}^{N_1+1}$.
We have
\begin{equation}\label{B0B1}
B_0(X,Y)=B_1(X,Y)+\sqrt{c}\langle X,Y\rangle\nu,~\forall ~X,~Y\in TM_1.\end{equation}
From  \eqref{sum} and \eqref{B0B1}, we get
\begin{equation}\label{sum2}
\begin{aligned}
F=\sum_{i,\beta}\Big(\l B_1(e_\beta^1, e_{\beta}^1), B_1(e_{i}^1, e_{i}^1)\r-2||B_1(e_i^1, e_{\beta}^1)||^2+c|e_i^1|^2|e_\beta^1|^2-2c\l e_{i}^1, e_{\beta}^1\r^2\Big).
\end{aligned}
\end{equation}
Hence from the definition of stability, we obtain the following key Lemma:
\begin{lem}\label{lem2.1}
Let $\Sigma$ be a compact minimal submanifold in a Riemannian product manifold $M=M_1\times M_2$, where $M_1$ is an $m_1$-dimensional
compact submanifold in a real space form $\mathbb{R}^{N_1}(c)~(c\geq 0)$, and $M_2$ is any Riemannian manifold. $F$ is given by \eqref{sum2}. If  $F>0$, for any
local orthonormal frame $\{e_1,\ldots,e_{n+p}\}$ on $M$, where $\{e_1,\ldots,e_n\}=\{e_i\}$ is tangent to $\Sigma$ and $\{e_{n+1},\ldots,e_{n+p}\}=\{e_\alpha\}$
is normal to $\Sigma$, then $\Sigma$ is unstable.
If $\Sigma$ is stable, then we have $\int_{\Sigma}Fd{\Sigma}\leq0$.
\end{lem}

\begin{rem}
In Lemma \ref{lem2.1},
if $M_2$ vanishes, then the ambient space is a compact submanifold $M$ in  a real space form $\mathbb{R}^{n}(c)$ with constant sectional curvature
$c ~(c\geq 0)$, we have $e_{\alpha}^1=e_{\alpha},~ e_i^1=e_i$. Denote by $B$ the second fundamental form
of the immersion $M\to \mathbb{R}^{n}(c)$, then
\begin{align*}
F=\sum_{i, \beta}\big\{\l B(e_i, e_i), B(e_{\beta}, e_{\beta})\r-2|| B(e_i, e_{\beta})||^2\big\}+n(m-n)c.
\end{align*}
Therefore, when $M_2$ vanishes, Lemma \ref{lem2.1} becomes a well-known result due to Lawson and Simons \cite{LS73}.
\end{rem}

As we mentioned in the beginning of this section, for any tangent vector $v$ at a point $(p_1, p_2)\in M_1\times M_2$, we have the decomposition $v=(v^1, v^2)$, where $v^i$ is tangent to $M_i$ at $p_i$. The
product structure $P$ on the product space $M=M_1\times M_2$ is then defined by (see \cite{TU10})
$$P(v)=P(v^1, v^2)=(v^1, -v^2),~\forall~ v=(v^1, v^2)\in T_{(p_1,p_2)}M_1\times M_2.$$
It is clear that $P$ is a linear isometry; $P$ is parallel (i.e. $\nabla P=0$); $P^2=Id$ and $\tr P=m_1-m_2$.
From the definition of $P$, it is easy to see that $e_{i}^1=\frac{1}{2}(e_i+Pe_i)$ and $e_{\alpha}^1=\frac{1}{2}(e_{\alpha}+Pe_{\alpha})$. Hence we obtain the following lemma.
\begin{lem}\label{lem}{\rm(see \cite{TU10})}
Denote by $A$ the square $n$-matrix $A=(a_{ij})=\big(\l e_i, Pe_j\r\big)$ and by $B$ the square $p$-matrix $B=(b_{\alpha\beta})=\big(\l e_{\alpha}, Pe_{\beta}\r\big)$,
we have
\begin{enumerate}
\item $\l e_i^1, e_j^1\r=\frac{1}{2}(\delta_{ij}+a_{ij}),\quad \l e_{\alpha}^1, e_{\beta}^1\r=\frac{1}{2}(\delta_{\alpha\beta}+b_{\alpha\beta}), \quad \l e_i^1, e_{\beta}^1\r=\frac{1}{2}\l e_i, Pe_{\beta}\r.$
\item {\rm$\sum\limits_i|e_i^1|^2=\frac{1}{2}(n+\tr A), \quad\sum\limits_{\beta}|e_{\beta}^1|^2=\frac{1}{2}(p+\tr B).$}
\item  {\rm$\sum\limits_{i,j}\l e_i^1, e_j^1\r^2=\frac{1}{4}(n+2\tr A+\tr A^2), \quad \sum\limits_{\alpha,\beta}\l e_{\alpha}^1, e_{\beta}^1\r^2=\frac{1}{4}(p+2\tr B+\tr B^2).$}
\item {\rm$\tr A+\tr B=2m_1-n-p.$}
\end{enumerate}
\end{lem}

\section{Proof of Theorem \ref{thm1.4}}
In this section, we present the proof of Theorem \ref{thm1.4}.

\begin{proof}[Proof of Theorem \ref{thm1.4}]
First, it is easy to check directly that all the submanifolds in the four cases in Theorem \ref{thm1.4} are stable minimal submainfolds.
In fact, $\Phi_1\times\Phi_2:\Sigma_1\times\Sigma_2\to M_1\times M_2$ is stable if and only if $\Phi_i:\Sigma_i\to M_i, i=1,2$, are stable (see \cite{TU10}
for details).

Conversely, assume that $\Phi=(\phi, \psi): \Sigma\to M=M_1\times M_2$ is a stable minimal immersion, where $M_1$ is a complete connected hypersurface in the Euclidean space $\mathbb{R}^{m_1+1}$, and $M_2$ is any Riemannian manifold. At a given point $p_1\in M_1$,
we choose an orthonormal frame $\{\et_1,\ldots,\et_{m_1},\et_{m_1+1}\}$ in $\mathbb{R}^{m_1+1}$  such that $\{\et_1,\ldots,\et_{m_1}\}$ is a local frame
tangent to $M_1$ and at $p_1\in M_1$, we have
$$B_1(\et_r, \et_s)=\la_r\delta_{rs}\et_{m_1+1}, ~\forall ~1\leq r,s\leq m_1,$$
where $\{\la_r\}$ are the principal curvatures of $M_1$ in $\mathbb{R}^{m_1+1}$ corresponding to the principal directions $\{\et_r\}.$
From the Gauss equation \eqref{Gauss}, at $p_1\in M_1$, we have
\begin{equation}\label{R1}\l R^1(\et_r, \et_s)\et_s,\et_r\r=\la_r\la_s~(r\neq s).\end{equation}

Assume that the sectional curvature $K_{M_1}$ of $M_1$ satisfies the following pinching condition:
$$\ep^2\leq K_{M_1}\leq 1$$
 for some constant $\ep\in(0, 1)$, then from \eqref{R1} we have
 \begin{equation}\label{lars}\ep^2\leq \la_r\la_s\leq 1 (r\neq s).\end{equation}
Without loss of generality, we assume that
\begin{equation}\label{paixu}0<\la_1\leq\cdots\leq\la_{m_1}.\end{equation}
Since $m_1\geq3$, from \eqref{lars}-\eqref{paixu}, it is easy to see that (cf. Lemma 3.2 in \cite{SH01})
\begin{equation}\label{eigen-ineq}\ep^2\leq \la_1\leq 1, \quad\ep\leq \la_2\leq\cdots\leq\la_{m_1-1}\leq 1, \quad\ep\leq \la_{m_1}\leq 1/\ep.\end{equation}
For $1\leq i\leq n, n+1\leq\alpha\leq n+p, 1\leq r,s\leq m_1,$ we denote by
\begin{equation*}
 c_{i}^{r}=\l e_i, \et_r\r=\l e_i^1, \et_r\r, ~ c_{\alpha}^{s}=\ \l e_{\alpha}, \et_s\r=\l e_{\alpha}^1, \et_s\r.
\end{equation*}
Since $\{\et_1,\ldots,\et_{m_1}\}$ is a local orthonormal  basis  of $T_{p_1}M_1$ and $\{e_1^1,\ldots,e_{n+p}^1\}$ is a spanning set of $T_{p_1}M_1$,
we have the following identities:
\begin{equation}\label{relation}
\sum\limits_{i=1}^nc_i^rc_i^s+\sum\limits_{\alpha=n+1}^{n+p} c_\alpha^rc_\alpha^s=\l\et_r,\et_s\r=\delta_{rs},~\sum\limits_{r=1}^{m_1}c_i^rc_j^r=\l e_i^1,e_j^1\r,~\sum\limits_{r=1}^{m_1}c_\alpha^rc_\beta^r=\l e_\alpha^1,e_\beta^1\r.
\end{equation}
From \eqref{sum} and \eqref{relation} we obtain that
\begin{align*}
F&=\sum_{i,\alpha,r, s}\Big(\la_r\la_s( c_{\alpha}^{r})^2( c_{i}^{s})^2-2\la_r\la_s c_{\alpha}^{r} c_{\alpha}^{s} c_{i}^{r} c_{i}^{ s}\Big)\\
&=-\sum_{i,\alpha,r}\la_r^2( c_{\alpha}^{r})^2( c_{i}^{r})^2+\sum_{r\neq s}\la_r\la_s\Big((\sum_i c_{i}^{r} c_{i}^{s})^2+(\sum_{\alpha} c_{\alpha}^{ r} c_{\alpha }^{s})^2+\sum_{i,\alpha}( c_{\alpha}^{r})^2( c_{i}^{s})^2\Big)\\
&=-\sum_{i,\alpha,r}\la_r^2( c_{\alpha}^{r})^2( c_{i}^{r})^2+\frac{1}{2}\sum_{r\neq s}\la_r\la_sG(r, s)\\
&=-\sum_{i,\alpha,r}\la_r^2( c_{\alpha}^{r})^2( c_{i}^{r})^2+\frac{1}{2}\sum_{r\neq s}(\la_r\la_s-\ep^2)G(r, s)+\frac{1}{2}\ep^2\sum_{r\neq s}G(r, s),
\end{align*}
where
\begin{align*}
G(r, s)=2(\sum_i c_{i}^{r} c_{i}^{s})^2+2(\sum_{\alpha} c_{\alpha}^{ r} c_{\alpha}^{ s})^2+\sum_{i,\alpha}(( c_{\alpha}^{ r})^2( c_{i}^{ s})^2+( c_{\alpha}^{ s})^2( c_{i}^{ r})^2).
\end{align*}
Obviously, $G(r, s)=G(s, r)\geq0.$
By \eqref{relation} and Lemma \ref{lem}, we have
\begin{align*}
\frac{1}{2}\sum_{r, s}G(r, s)=&\sum_{i,j,r,s} c_{i}^{r} c_{i}^{s} c_{j}^{r} c_{j}^{s}+\sum_{\alpha,\beta,r,s} c_{\alpha}^{ r} c_{\alpha}^{ s} c_{\beta}^{ r} c_{\beta}^{ s}+\sum_{ \alpha,r}( c_{\alpha}^{ r})^2\sum_{ i,s}( c_{i}^{ s})^2\\
=&\sum_{i, j}\l e_i^1, e_j^1\r^2+\sum_{\alpha,\beta}\l e_{\alpha}^1, e_{\beta}^1\r^2+(\sum_{\alpha}|e_{\alpha}^1|^2)(\sum_i|e_i^1|^2)\\
=&\frac{1}{4}\big[(n+2\tr A+\tr A^2)+(p+2\tr B+\tr B^2)+(\tr A+n)(\tr B+p)\big]\\
=&\frac{1}{4}\big[2(2m_1-n)+2\tr A^2+(\tr A+n)(2m_1-n-\tr A)\big]
\end{align*}
and
\begin{align*}
\frac{1}{2}\sum_rG(r, r)=&\sum_r(\sum_i( c_{i}^{r})^2)^2+\sum_r(\sum_{\alpha}( c_{\alpha}^{ r})^2)^2+\sum_{ i, \alpha,r}( c_{\alpha }^{r})^2( c_{i}^{r})^2\\
=&\sum_r\Big[\big(\sum_i( c_{i}^{r})^2+\sum_{\alpha}( c_{\alpha}^{ r})^2\big)^2-(\sum_{i}( c_{i}^{r})^2)(\sum_{\alpha}( c_{\alpha}^{ r})^2)\Big]\\
=&m_1-\sum_{i, \alpha,r}( c_{\alpha }^{r})^2( c_{i}^{r})^2.
\end{align*}
Hence
\begin{align*}
\frac{1}{2}\ep^2\sum_{r\neq s}G(r, s)=&\frac{\ep^2}{2}\sum_{r, s} G(r, s)-\frac{\ep^2}{2}\sum_rG(r, r)\\
=&\frac{\ep^2}{4}\big[2(2m_1-n)+2\tr A^2+(\tr A+n)(2m_1-n-\tr A)\big]\\
&-m_1\ep^2+\ep^2\sum_{i, \alpha, r}( c_{\alpha}^{ r})^2( c_{i}^{ r})^2,
\end{align*}
which implies that
\begin{equation}
\label{F1-3}
\begin{aligned}
F&=\frac{\ep^2}{4}\big[(n+\tr A)(2m_1-n-\tr A)-2n+2\tr A^2\big ]\\
+&\sum_{i,\alpha, r}(\ep^2-\la_r^2)( c_{\alpha}^{ r})^2( c_{i}^{ r})^2
+\frac{1}{2}\sum_{r\neq s}(\la_r\la_s-\ep^2)G(r, s).
\end{aligned}
\end{equation}
Denote by $$\begin{aligned}F_1&=\frac{\ep^2}{4}\big[(n+\tr A)(2m_1-n-\tr A)-2n+2\tr A^2\big ],\\
F_2&=\sum_{i,\alpha, r}(\ep^2-\la_r^2)( c_{\alpha}^{ r})^2( c_{i}^{ r})^2,~
F_3=\frac{1}{2}\sum_{r\neq s}(\la_r\la_s-\ep^2)G(r, s)\geq 0,
\end{aligned}
$$
we have $F=F_1+F_2+F_3$.
By the Cauchy-Schwarz inequality, we have
$
\tr A^2\geq \frac{(\tr A)^2}{n},
$
and the equality holds if and only if $A=\la I_n$ for certain function $\la$.
This implies
\begin{equation}
\label{F1}
F_1\geq \frac{\ep^2}{4}\left[(n+\tr A)(2m_1-n-\tr A)-2n+\frac{2(\tr A)^2}{n}\right ].
\end{equation}
Since $\la_r\leq 1/\ep,~\forall ~1\leq r\leq m_1$ (see \eqref{eigen-ineq}), we obtain that
\begin{align}
F_2=&\sum_{i, \alpha, r}(\ep^2-\la_r^2)( c_{\alpha}^{ r})^2( c_{i}^{ r})^2\geq
\sum_{i, \alpha, r}(\ep^2-\frac{1}{\ep^2})( c_{\alpha}^{ r})^2( c_{i}^{ r})^2\label{ineq0}\\
=&\sum_{r}(\ep^2-\frac{1}{\ep^2})(1-\sum_j( c_{j}^{ r})^2)\sum_i( c_{i}^{ r})^2\nonumber\\
=&(\ep^2-\frac{1}{\ep^2})\sum_{i,r}( c_{i}^{ r})^2-(\ep^2-\frac{1}{\ep^2})\sum_{r}(\sum_i( c_{i}^{ r})^2)^2\nonumber\\
\geq&(\ep^2-\frac{1}{\ep^2})\frac{\tr A+n}{2}-(\ep^2-\frac{1}{\ep^2})\frac{1}{m_1}(\sum_{i,r}( c_{i}^{ r})^2)^2\label{ineq1}\\
=&(\ep^2-\frac{1}{\ep^2})\frac{\tr A+n}{2}-(\ep^2-\frac{1}{\ep^2})\frac{1}{m_1}(\frac{\tr A+n}{2})^2.\nonumber
\end{align}
So we get
\begin{equation}\label{F-formula}
\begin{aligned}
F\geq&\frac{\ep^2}{4}(\tr A+n)\left[(\frac{2-n}{n}\tr A+2m_1-n-2)\right.\\
&\left.+2(1-\frac{1}{\ep^4})-(1-\frac{1}{\ep^4})\frac{1}{m_1}(\tr A+n)\right]\\
=&\frac{\ep^2}{4}(\tr A+n)f(\tr A),
\end{aligned}
\end{equation}
where $$f(x)=\left(\frac{2-n}{n}-(1-\frac{1}{\ep^4})\frac{1}{m_1}\right)x+2m_1-n-2+(1-\frac{1}{\ep^4})(2-\frac{n}{m_1}).$$

Set $\ep^2=\frac{1}{\sqrt{m_1-1}}.$ We discuss the following two different cases:

\textbf{Case 1}. $n\leq m_1.$
By a direct computation we have
$$f(-n)=0,\quad f(n)=\frac{4}{m_1}(m_1-n)\geq0.$$
Since $-n\leq\tr A\leq n$,  we have $\tr A+n\geq0$ and $f(\tr A)\geq\min\{f(-n), f(n)\}=0,$ which together with \eqref{F-formula} implies that
$F\geq 0$. On the other hand, since $\Sigma$ is stable,
from Lemma \ref{lem2.1} we know that $\int_{\Sigma}F d{\Sigma}\leq 0$, hence we get $F=0$. Then all the equalities hold in the above inequalities. That is, we have $F_3=0$, all the equalities in \eqref{F1}-\eqref{ineq1}  hold and $(\tr A+n)f(\tr A)=0$.

 ``=" in (\ref{F1}) implies that $A=\la I_n, \la\in[-1, 1].$
Note that (see Lemma \ref{lem}) $$\sum_{i,r}( c_{i}^{ r})^2=\sum\limits_i|e_i^1|^2=\frac{n+\tr A}{2}=\frac{n(1+\la)}{2},$$
``=" in (\ref{ineq1}) then implies that
$$\sum_i( c_{i}^{1})^2=\cdots=\sum_i( c_{i}^{m_1})^2=\frac{n(1+\la)}{2m_1}.$$
Since $\la_r\la_s\leq 1 (r\neq s)$, there
exists at least one $r$ ($1\leq r\leq m_1$)  such that $\la_r<\frac{1}{\ep}$, then ``=" in (\ref{ineq0}) implies that for this $r$, we have
$$
\begin{aligned}
0&=(\sum_{\alpha}( c_{\alpha}^{ r})^2)(\sum_{i}( c_{i}^{ r})^2)=(1-\sum_{j}( c_{j}^{ r})^2)(\sum_{i}( c_{i}^{ r})^2)\\
&=\frac{n(1+\la)(2m_1-n-n\la)}{4m_1^2}\geq
\frac{n(1+\la)(2n-n-n\la)}{4m_1^2}\geq0,
\end{aligned}
$$
from which we conclude that $\la^2=1, A=\pm I_{n}$ and $
n=m_1$ if $A=I_{n}$.

So finally we obtain that either  $A=I_n$ and $n=m_1$ or $A=-I_n$.

\textbf{Case 2}. $n>m_1.$ We adapt an argument similar to that in \cite{TU10}.  For any point $x\in\Sigma$, we have that $\dim\ker d\phi_x\geq n-m_1$, hence there exists an $(n-m_1)$-dimensional linear subspace $V_x\subset\ T_x\Sigma$ such that $Pv=-v$ for any $v\in V_x$. Thus, we can decompose $T_x\Sigma=V_x\oplus Z_x$, where $Z_x$ is orthogonal to $V_x$ and $\dim Z_x=m_1$. Then the matrix $A$ can be written as
$$A=\left(\begin{array}{c|c}-I_{n-m_1} & 0 \\&\\\hline \\0 & \,\,\tilde{A}\end{array}\right),$$
 with $\hat{A}_{ij}=\langle Pz_i,z_j\rangle$, and $\{z_1,\dots,z_p\}$ being
an orthonormal basis of $Z_x$. Hence,
\begin{equation}
\label{tildeA}
\tr A=\tr \tilde{A}+m_1-n,\, \tr A^2=\tr \tilde{A}^2+n-m_1,
\end{equation}
and by the Cauchy-Schwarz inequality, we have
\begin{align}\label{Cau2}
\tr \tilde{A}^2\geq\frac{(\tr\tilde{A})^2}{m_1},
\end{align}
where the equality holds if and only if $\tilde{A}=\la I_{m_1}$ for certain function $\la$.
Using an analogous argument to that in Case 1, from \eqref{ineq0}-\eqref{Cau2}, we get
\begin{align*}
F\geq&\frac{\ep^2}{4m_1}(m_1-1-\frac{1}{\ep^4})(m_1^2-(\tr\tilde{A})^2)=0,
\end{align*}
where we used the fact that $\ep^2=\frac{1}{\sqrt{m_1-1}}$.
 On the other hand, since $\Sigma$ is stable,
from Lemma \ref{lem2.1} we know that $\int_{\Sigma}F d{\Sigma}\leq 0$, hence we get $F=0$.
Then all the equalities hold in the above inequalities.
That is, we have that $F_3=0$ and all the equalities in \eqref{ineq0},  \eqref{ineq1} and \eqref{Cau2} hold.

``=" in (\ref{Cau2}) implies $\tilde{A}=\la I_{m_1}.$ Since  $-m_1\leq \tr \tilde{A}\leq m_1,$ we have  $-1\leq\la\leq1$.
Note that (see Lemma \ref{lem}) $$\sum_{i,r}( c_{i}^{ r})^2=\sum\limits_i|e_i^1|^2=\frac{n+\tr A}{2}=\frac{m_1+\tr \tilde{A}}{2}=\frac{m_1(1+\la)}{2},$$
``=" in (\ref{ineq1}) then implies that
$$\sum_i(c_{i}^{1})^2=\cdots=\sum_i(c_{i}^{m_1})^2=\frac{1+\la}{2}.$$
Since $\la_r\la_s\leq 1 (r\neq s)$, there
exists at least one $r$ ($1\leq r\leq m_1$)  such that $\la_r<\frac{1}{\ep}$, then ``=" in (\ref{ineq0}) implies that for this $r$, we have
$$0=(\sum_{\alpha}( c_{\alpha}^{ r})^2)(\sum_i( c_{i}^{ r})^2)=(1-\sum_{i}( c_{i}^{ r})^2)(\sum_i( c_{i}^{ r})^2)=\frac{1-\la^2}{4}.$$
So we get that $\la^2=1, \tilde{A}=\pm I_{m_1}.$
Finally, we get that either $A=-I_n$ or $A=-I_{n-m_1}\oplus I_{m_1}.$

In summary, if $\Psi$ is stable, we obtain that the matrix $A$ has only three possibilities:
$${\rm(i)}A=I_{n}\mbox{ and } n=m_1,  \quad {\rm(ii)} A=-I_n, \quad {\rm(iii)} A=-I_{n-m_1}\oplus I_{m_1}.$$
After an analogous argument to that in \cite{TU10}, we obtain that (i) $A=I_{n}\mbox{ and } n=m_1$ leads to case (1) in Theorem \ref{thm1.4}, (ii) $A=-I_n$ leads to cases (2) and (3) in Theorem \ref{thm1.4}, and we derive case (4) in Theorem \ref{thm1.4} from (iii)$A=-I_{n-m_1}\oplus I_{m_1}$.
Hence we  have proved the first part of Theorem \ref{thm1.4}.

In particular, when the ambient space is  an $m$-dimensional $(m\geq 3)$ complete   hypersurface $M$  in  $\mathbb{R}^{m+1}$,
 we regard $M$ as a Riemannian product of $M_1=M$ and a point $p_2$. We use proof by contradiction.
 Assume that $\Sigma$ is a stable minimal submanifold in $M$,
we use the same notations as in \eqref{F1-3}, in this case, it is easy to see that $\dim M=m=m_1>n$ and $A=I$.  Set $\ep^2=\frac{1}{\sqrt{m+1}}$, from \eqref{F-formula} we obtain that
$F\geq 0$. Using a same argument as in the proof of Case 1 above, we obtain that $n=m$, which contradicts the fact that $\dim M=m>n$.
Hence, we conclude that in this case there exist no stable minimal submanifolds in $M$.
This completes the proof of Theorem \ref{thm1.4}.
\end{proof}

%-------------------------------------------------
\bibliographystyle{amsplain}
\bibliography{ref-stable-min}

\end{document}